\documentclass{article}
\usepackage{graphicx, amsmath, amsfonts, amssymb, amsthm, verbatim, float, mathtools, tikz}
\usetikzlibrary{decorations.pathreplacing, calligraphy}

\newtheorem{theorem}{Theorem}
\newtheorem{lemma}[theorem]{Lemma}
\newcommand{\stirling}{\genfrac\{\}{0pt}{}}
\makeatletter
\def\underbrace#1{%
  \@ifnextchar_{\tikz@@underbrace{#1}}{\tikz@@underbrace{#1}_{}}}
\def\tikz@@underbrace#1_#2{%
  \tikz[baseline = (a.base)] {\node[inner sep = 2] (a) {\(#1\)};
  \draw[line cap = round,
        decorate,
        decoration = {calligraphic brace, amplitude = 4pt},
        line width = 1pt]
    (a.south east) -- node[pos = 0.5, below, inner sep = 7pt] {\(\scriptstyle #2\)} (a.south west);}}
\def\overbrace#1{%
  \@ifnextchar^{\tikz@@overbrace{#1}}{\tikz@@overbrace{#1}^{}}}
\def\tikz@@overbrace#1^#2{%
  \tikz[baseline = (a.base)] {\node[inner sep = 2] (a) {\(#1\)};
  \draw[line cap = round,
        decorate,
        decoration = {calligraphic brace, amplitude = 4pt},
        line width = 1pt]
    (a.north west) -- node[pos = 0.5, above, inner sep = 7pt] {\(\scriptstyle #2\)} (a.north east);}}
\makeatother

\begin{document}

\title{On raw moments of the binomial distribution}
\author{Kalle Leppälä$^1$}
\date{%
{\small $^1$\! The Organismal and
Evolutionary Biology Research Programme, University of Helsinki, Viikinkaari 1 PL 65, 00014 Helsinki, Finland} \vspace{0.4cm}
\\
\today
}
\maketitle

\begin{abstract}
We study the $k$:th raw moment of a variable $R$ following the binomial distribution $\text{B}(n,p)$, where $n/k \rightarrow \beta > 0$.
It is known that $\mathbb{E}(R^k)$ is bounded both from below and from above by functions of the form $k^k \Psi^k$. We solve the asymptotically optimal value of $\Psi$ as a function of $p$ and $\beta$.
\end{abstract}

\section{Introduction}
If $n$ balls are randomly painted red with probability $p$ and blue otherwise, the number of red balls $R$ follows the binomial distribution $\text{B}(n, p)$. Ahle \cite{ahle2022sharp} describes an interesting way to look at the $k$:th raw moment $\mathbb{E}(R^k)$. If the $n$ balls are next placed in an urn, and a sample $k$ balls is randomly drawn with replacement, then
\begin{equation} \label{LHS}
\mathbb{P}(\text{``the sample is all red''}) = \sum_{i=0}^n
\,\underbrace{\!\!\displaystyle{\binom{n}{i}p^i(1-p)^{n-i}}\!\!}_{= \, \mathbb{P}(R = i)}\,
\Big(\frac{i}{n}\Big)^{\!k}  = \frac{\mathbb{E}(R^k)}{n^k}\,.
\end{equation}

\vspace*{-0.25cm}
\noindent Obviously $\mathbb{E}(R^k) \le n^k$ because the probability is at most one, and using Jensen's inequality with $\mathbb{E}(R) = np$ gives a trivial lower bound $\mathbb{E}(R^k) \ge (np)^k$.
We can also change perspective from the number $R$ of red balls to the sample size $S$, paying no mind to how many times each ball was drawn.
Using the notation $\stirling{k}{j}$ for the Stirling numbers of the second kind \cite{sagan2020combinatorics}, that is, the number of ways $k$ labeled elements can be divided into $j$ unlabeled nonempty subsets, and $(n)_j = n(n-1) \cdots (n-j+1)$ for the falling factorial, we have
\begin{equation} \label{RHS}
\mathbb{P}(\text{``the sample is all red''}) = \sum_{j=1}^{\min\{k, n\}}
\!\!\!\underbrace{\!\!\displaystyle{\stirling{k}{j} \frac{(n)_j}{n^k}}\!\!}_{= \, \mathbb{P}(S = j)}\!
p^{\,j} = \mathbb{E}(p^S) \,.
\end{equation}

\vspace*{-0.25cm}
\noindent Together \eqref{LHS} and \eqref{RHS} create a well known \cite{knoblauch2008closed} formula for the raw moments of the binomial distribution. Estimating $p^{\,j} \le p$ when $j \ge 1$ 
gives a slightly less trivial upper bound $\mathbb{E}(R^k) \le n^k p$, and because expectation is linear we can compute $\mathbb{E}(S) = n \mathbb{E}(\text{``first ball is in the sample''}) = n(1-(1-1/n)^k)$ to get a moderately less trivial lower bound $\mathbb{E}(R^k) \ge n^kp^{n(1-(1-1/n)^k)}$ with Jensen's inequality.
The new bound is easily verified as an improvement by using the binomial approximation $(1-1/n)^k \ge 1-k/n$. Because $(1-1/n)^n \rightarrow 1/e$, for any $\varepsilon > 0$ and $n$ big enough we now have the bounds
\[
n^k p^{n(1-(e+\varepsilon)^{-\frac{k}{n}})}\le \mathbb{E}(R^k) \le n^k p\,.
\]
If $n/k \rightarrow \infty$, then by l'H\^opital's rule the left hand side is approximately the trivial bound $(np)^k$. If $n/k \rightarrow \beta$ for some positive and finite $\beta$, then the left hand side is approximately $n^k p^{\beta(1-e^{-1/\beta})k}$. Finally, if $n/k \rightarrow 0$, then the left hand side is approximately $n^k p^n$.
In each of the three instances the simple bounds leave us wanting for a more precise description of the asymptotic behaviour of the raw moment $\mathbb{E}(R^k)$.
General results \cite{latala1997estimation} on independent and identically distributed random variables are not enough for the task either.

Using sharp bounds of the Lambert function, the inverse of $xe^x$, Ahle \cite{ahle2022sharp} derives the upper bound
\begin{equation} \label{ahle}
\mathbb{E}(R^k) \le \left(\frac{k}{\log(1+ k/(np))} \right)^{\!k}.
\end{equation}
As noted in the article, this is less than $(np)^k(1+k/(2np))^k$ and so the bound \eqref{ahle} tightly matches the trivial lower bounds when $n/k \rightarrow \infty$. However, in the cases $n/k \rightarrow \beta>0$ and $n/k \rightarrow 0$ the bound \eqref{ahle} is no longer sharp; indeed it doesn't always even beat the trivial upper bound $pn^k$. 
In the work at hand we investigate the second case $n/k \rightarrow \beta > 0$, determining in Theorem \ref{main} the precise asymptote of $\log\big( \mathbb{E}(R^k)/k^k \big)$.
Unfortunately the result is not effective, in other words, we could not pin down the speed of convergence towards the asymptote.

\begin{theorem} \label{main}
Suppose $n = \beta k + o(k)$. Then
\[
\mathbb{E}(R^k) = k^k (\Psi + o(1))^k \,,
\]
where
\begin{equation} \label{asymptote}
\Psi = \frac{\beta^\beta(e^{\chi} - 1)^{\tau}}{\tau^{\tau} (\beta - \tau)^{\beta - \tau}\chi} \,,
\end{equation}
\[
\chi = \frac{1}{\beta} + W_0 \Big( \frac{e^{-\frac{1}{\beta}}(1-p)}{\beta p}\Big) \,, \quad
\tau = \frac{\beta}{1-p}\Bigg(\frac{1}{\beta W_0\Big(e^{-\frac{1}{\beta}}(1-p)/(\beta p)\Big)+1} -p\Bigg)\,,
\]
and $W_0$ stands for the principal branch of the Lambert function.
\end{theorem}

\section{Proof of the asymptote}

To deal with the values of $j$ in \eqref{RHS} close to extremities, we shall need the unimodality of the sequence of coefficients. In his article about unimodality for the Stirling numbers of the second kind \cite{dobson1968note}, Dobson points out an earlier unpublished proof by D. Klarner showing logarithmic concavity, a stronger property that implies unimodality when the sequence is also positive. Here as well this approach turns out to be fruitful: the desired result is a trivial corollary of Klaner's inequality \eqref{log-concave}, which we also prove for the sake of completeness.

\begin{lemma}[Klaner] \label{Klaner}
For all $2 \le j \le k-1$,
\begin{equation} \label{log-concave}
j\frac{\stirling{k}{j}}{\stirling{k}{j-1}} \ge
(j+1)\frac{\stirling{k}{j+1}}{\stirling{k}{j}}\,.
\end{equation}
Consequently, the sequences
\[
\stirling{k}{j}\,, \quad \stirling{k}{j}(n)_j \quad \text{and} \quad \stirling{k}{j}(n)_j p^j\,,
\]
where $1 \le j \le \min\{k, n\}$, are each logarithmically concave and as positive sequences thereby also unimodal.
\end{lemma}
\begin{proof}
When $k=3$, the inequality \eqref{log-concave} is valid as $2 \cdot 3/1 \ge 3\cdot 1/3$. We make an induction hypothesis and proceed to show it for $k+1$. Using the recurrence \cite{sagan2020combinatorics}
\[
\stirling{k+1}{j} = \stirling{k}{j - 1} + j \stirling{k}{j}
\]
and the induction hypothesis, we obtain
\begin{align*}
j\frac{\stirling{k+1}{j}}{\stirling{k+1}{j-1}}
& = j\frac{\stirling{k}{j-1} + j\stirling{k}{j}}{\stirling{k}{j-2} + (j-1) \stirling{k}{j-1}} = j\frac{ 1+ j\stirling{k}{j}/ \stirling{k}{j-1}}{\tfrac{j-1}{j}\stirling{k}{j-1}/\stirling{k}{j} + j-1} \\ &
= \frac{j}{j-1} \cdot \frac{1+ j\stirling{k}{j}/ \stirling{k}{j-1}}{\tfrac{1}{j}\stirling{k}{j-1}/ \stirling{k}{j} + 1}
\ge \frac{j+1}{j} \cdot \frac{1+ j\stirling{k}{j}/ \stirling{k}{j-1}}{\tfrac{1}{j}\stirling{k}{j-1}/ \stirling{k}{j} + 1} \\
& =(j+1)\frac{1 + (j+1)\tfrac{j}{j+1}\stirling{k}{j}/\stirling{k}{j-1}}{\stirling{k}{j-1}/\stirling{k}{j} + j}
=(j+1)\frac{\stirling{k}{j} + (j+1)\stirling{k}{j+1}}{\stirling{k}{j-1} + j\stirling{k}{j}} \\
& = (j+1) \frac{\stirling{k+1}{j+1}}{\stirling{k+1}{j}} \,.
\end{align*}
Thus, \eqref{log-concave} is valid by induction. Because $j+1\ge j$ we get
\[
\frac{\stirling{k}{j}}{\stirling{k}{j-1}} \ge
\frac{\stirling{k}{j+1}}{\stirling{k}{j}} \,,
\]
in other words, Stirling numbers of the second kind form a logarithmically concave sequence with respect to $j$. Now also
\[
\frac{\stirling{k}{j}(n)_j}{\stirling{k}{j-1}(n)_{j-1}} = (n-j+1)\frac{\stirling{k}{j}}{\stirling{k}{j-1}}
\ge
(n-j)\frac{\stirling{k}{j+1}}{\stirling{k}{j}} = 
\frac{\stirling{k}{j+1} (n)_{j+1}}{\stirling{k}{j}(n)_j}
\]
and
\[
\hspace{-0.3mm}
\frac{\stirling{k}{j}(n)_jp^{\,j}}{\stirling{k}{j-1}(n)_{j-1}p^{\,j-1}} = (n\hspace{-0.1mm}-\hspace{-0.1mm}j\hspace{-0.1mm}+\hspace{-0.1mm}1)p\frac{\stirling{k}{j}}{\stirling{k}{j-1}}
\ge
(n\hspace{-0.1mm}-\hspace{-0.1mm}j)p\frac{\stirling{k}{j+1}}{\stirling{k}{j}} = 
\frac{\stirling{k}{j+1} (n)_{j+1}p^{\,j+1}}{\stirling{k}{j}(n)_j p^{\,j}} \,. \qedhere
\]
\end{proof}

We can now prove the main result Theorem \ref{main}. When $n = \beta k + o(k)$, the coefficients in \eqref{RHS} behave nicely, attaining their maximum not too close to the extremities. The sum is then dominated by the terms around this maximum, as we show using essentially the Laplace's method.

\begin{proof}[Proof of Theorem \ref{main}.]
Temme \cite{temme1993asymptotic}, \cite[chapter 34]{temme2014asymptotic} uses a modified saddle point method to approximate the Stirling numbers of the second kind with
\begin{equation} \label{temme}
\stirling{k}{j} = \binom{k}{j} \left( \frac{k-j}{e}\right)^{k-j}\frac{(e^{\chi} - 1)^j}{\chi^{k}}\sqrt{\frac{k-j}{k(\chi - k/j + 1)}}(1+o(1))\,,
\end{equation}
where $\chi$ is the real and positive saddle point of the function 
\[
\phi(x) = -k\log x + j\log(e^x - 1)\,. 
\]
Differentiating yields
\[
\phi'(x) = -\frac{k}{x} + \frac{je^x}{e^x - 1}\,,
\]
and so
\begin{equation} \label{chi}
\frac{k}{\chi} = \frac{je^{\chi}}{e^{\chi} - 1} \quad \Leftrightarrow \quad
-\frac{k}{j}e^{-\frac{k}{j}} = \Big(\chi - \frac{k}{j}\Big)e^{\chi - \frac{k}{j}}\,.
\end{equation}
Because $-1/e < -ke^{-k/j}/j < 0$, equation \eqref{chi} has two real solutions corresponding to the two real branches of the Lambert function. The branch $W_{-1}$ gives the invalid solution $\chi = 0$, so we should express the positive real saddle point $\chi$ using the principal branch $W_0$ as
\[
\chi = W_0\Big( -\frac{k}{j} e^{-\frac{k}{j}}\Big) + \frac{k}{j}\,.
\]
With this we can work out
\[
\sqrt{\frac{k-j}{k(\chi - k/j + 1)}} \le 1 \quad \Leftrightarrow \quad -\frac{j}{k} \le W_0\Big( -\frac{k}{j} e^{-\frac{k}{j}}\Big) \quad \Leftrightarrow \quad
-\frac{j}{k}e^{-\frac{j}{k}} \le -\frac{k}{j}e^{-\frac{k}{j}} \,.
\]
The last inequality is easily verified for $1 \le j \le k$, thus the square root factor in \eqref{temme} is bounded between zero and one.

The convergence of $o(1)$ to zero as $k$ tends to infinity in \eqref{temme} is uniform with respect to $j$ as long as $\delta k < j < (1-\delta)k$, where $\delta$ is any positive constant
\cite{chelluri2000asymptotic}, see \cite{connamacher2020uniformity}
for a proof in a more general context of $r$-associated Stirling numbers of the second kind. A recent preprint \cite{canfield2024uniform} of Canfiel, Helton and Hughes actually demonstrates uniform convergence without restrictions, but in our case the extremities of $j$ require separate attention anyway, so we write $j = \tau k$, $\delta < \tau < 1-\delta$.
An effective version $k^k/e^{k-1} \le k! \le k^{k+1}/e^{k-1}$ \cite{knuth1997art} of Stirling's approximation applied to the binomial coefficient in \eqref{temme} gives
\[
\log \left( \stirling{k}{j} \right)
= (1-\tau)k\log k + \big(-\tau \log \tau - (1-\tau) + \tau\log(e^{\chi}-1) - \log \chi\big)k + o(k)\,.
\]
The convergence in $o(k)$ remains uniform with respect to $\tau$.
For the falling factorials we write $n = \beta k + o(k)$, $(1-\delta)\min\{1,\beta\} > \tau$,
and continue with Stirling's approximation to get
\[
\log \big( (n)_j \big) = \tau k\log k + \big( \beta \log \beta - (\beta - \tau)\log(\beta - \tau) - \tau \big)k + o(k)\,,
\]
altogether
\[
\log \left( \stirling{k}{j}(n)_{j}p^{j} \right) = k \log k + \psi(\tau) k + o(k)\,,\]\[
\psi(\tau) = \tau \log(e^{\chi} - 1) - \log \chi 
+ \tau \log \Big( \frac{\beta}{\tau} - 1\Big)
- \beta \log \Big( 1 - \frac{\tau}{\beta}\Big) + \tau \log p -1 \,.
\]
Here $\chi = W_0(-e^{-1/\tau}/\tau) + 1/\tau$ can't be expressed without the Lambert function, so we use the substitution $\tau = (1-e^{-\chi})/\chi$ \eqref{chi} and maximize the function $\psi(\tau)$ in terms of $\chi$.
By differentiating and using \eqref{chi} we get
\begin{align*}
\psi'(\tau) &  = \!\!\!\!\!\!\! \underbrace{\!\!\displaystyle{\log(e^{\chi} - 1)}\!\!}_{= \, \log \tau + \chi + \log \chi} \!\!\!\!\!\!\! + \underbrace{\!\!\displaystyle{\frac{\tau e^{\chi}\chi'}{e^{\chi} - 1} - \frac{\chi'}{\chi}}\!\!}_{= \, 0}
-\log \tau + \log(\beta - \tau)
+ \log p \,. \\
& = \chi + \log \chi + \log(\beta - \tau) + \log p \,.
\end{align*}
Setting the derivative to zero gives
\[
e^{\chi} \chi (\beta - \tau)p = 1 \quad \Leftrightarrow \quad
\Big( \chi - \frac{1}{\beta} \Big)e^{\chi - \frac{1}{\beta}}
= \frac{e^{-\frac{1}{\beta}}(1-p)}{\beta p} \,.
\]
Because $\tau < \beta$ and $W_0(-e^{-1/\tau}/\tau) > -1$, we have $\chi -1/\beta > -1$ and so we can solve this by applying the principal branch of the Lambert function. The function $\psi(\tau)$ attains its maximum $\Psi = \psi(\tau_0)$ at the point
\[
\chi_0 = \frac{1}{\beta} + W_0 \Big( \frac{e^{-\frac{1}{\beta}}(1-p)}{\beta p}\Big) \,\, \Leftrightarrow \,\,
\tau_0 = \frac{\beta}{1-p}\Bigg(\frac{1}{\beta W_0\Big(e^{-\frac{1}{\beta}}(1-p)/(\beta p)\Big)+1} -p\Bigg)\,.
\]
If $\delta$ is small enough, this point is strictly inside the interval $(\delta, (1-\delta)\min\{1, \beta\})$.
We are ready to bound the sum from above and from below. For any $\varepsilon > 0$
\begin{align*}
\sum_{j=1}^{\min\{k,n\}} \stirling{k}{j}(n)_j p^{\,j} & \le
\min\{k, n\} \hspace{-0.218cm} \max_{\substack{j \ge 1 \\ j \le \min\{k, n\}}} \hspace{0.184cm} \stirling{k}{j}(n)_j p^{\,j} \\ & = 
\min\{k, n\} \hspace{-0.6cm} \max_{\substack{j \ge  \delta k  \\ j \le (1-\delta)\min\{k  , n \}}} \hspace{-0.2cm} \stirling{k}{j}(n)_j p^{\,j}  \\ & \le
\min\{k, n\} k^k \Psi^k (1+ \varepsilon)^k \le k^k (\Psi + \varepsilon)^k
\end{align*}
for all big enough $k$; the second step uses uses the uniform convergence of $o(k)$ with respect to $j$ and unimodality from Lemma \ref{Klaner}, and the third step uses the uniform convergence again. To clarify, the uniform convergence guarantees there is a local maximum strictly inside $(\delta k, (1-\delta)\min\{k, n\})$, and by unimodality this is also the global maximum.
As for the lower bound, using the uniform convergence
\[
\sum_{j=1}^{\min\{k,n\}} \stirling{k}{j}(n)_j p^{\,j} \ge
\hspace{-0.3cm} \sum_{j=\lceil \delta k \rceil}^{\lfloor (1-\delta)\min\{k,n\}\rfloor} \hspace{-0.1cm}  \stirling{k}{j}(n)_j p^{\,j}
\ge k^k (1 - \varepsilon)^k
\hspace{-0.4cm} \sum_{j=\lceil \delta k \rceil}^{\lfloor (1-\delta)\min\{k,n\}\rfloor} \hspace{-0.2cm}  \psi\Big(\frac{j}{k}\Big)^{\!k} 
\]
for any $\varepsilon > 0$ and all $k$ big enough.
We view the last sum as $k$ times the area of a ziggurat consisting of rectangles of width $1/k$ and height $\psi(j/k)^k$. Let $h$ be largest of these heights. The slope of $\psi(\tau)^k$ is $k\psi(\tau)^{k-1}\psi'(\tau)$. Therefore for any $\varepsilon > 0$ we eventually have $h \ge (1 - \varepsilon)\Psi^k$ for all big $k$. Restriction to $\tau \in (\delta, (1-\delta)\min\{1, \beta\})$ means that while $\psi(\tau)$ has vertical asymptotes at $\tau = 0$ and $\tau = \min\{1, \beta\}$, the value of $\lvert \psi'(\tau)\rvert$ is bounded from above by some finite constant $\Delta$. An isosceles triangle whose sides have slopes $\pm 2$ times the steepest slope of $\psi(\tau)^k$, and whose apex is at the maximal value of $\psi(j/k)^k$, is now completely inside the ziggurat by the mean value theorem and thereby has lesser area.
The construction is illustrated in Figure \ref{ziggurat}.
\begin{figure}[H]
\includegraphics{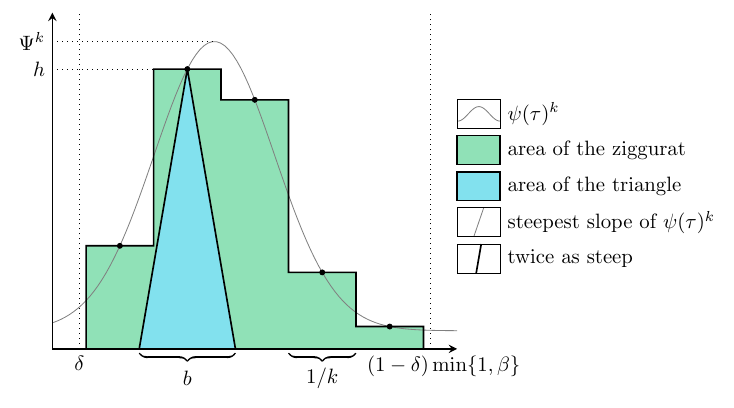}
\centering
\caption{The isosceles triangle is inside the ziggurat.}
\label{ziggurat}
\end{figure}
\noindent We can bound the base $b$ of the isosceles triangle from below as
\[
b = 2\frac{h}{2\,\displaystyle{\max_{\tau}}\{k \psi(\tau)^{k-1}\lvert\psi'(\tau)\rvert \}} \ge \frac{(1-\varepsilon) \Psi}{k\Delta} \,,
\]
and then we are ready to conclude the proof with
\[
\sum_{j=1}^{\min\{k,n\}} \stirling{k}{j}(n)_j p^{\,j}
\ge k^k(1-\varepsilon)^k k \frac{1}{2}bh
\ge k^k(\Psi - \varepsilon)^k \,. \qedhere
\]
\end{proof}

\bibliographystyle{siam}
\bibliography{bibliography}

\begin{thebibliography}{10}

\bibitem{ahle2022sharp}
{\sc T.~D. Ahle}, {\em Sharp and simple bounds for the raw moments of the binomial and poisson distributions}, Statistics \& Probability Letters, 182 (2022), p.~109306.

\bibitem{canfield2024uniform}
{\sc E.~R. Canfield, J.~W. Helton, and J.~A. Hughes}, {\em Uniform convergence of an asymptotic approximation to associated stirling numbers}, arXiv preprint arXiv:2409.01489,  (2024).

\bibitem{chelluri2000asymptotic}
{\sc R.~Chelluri, L.~Richmond, and N.~Temme}, {\em Asymptotic estimates for generalized stirling numbers}, Analysis, 20 (2000), pp.~1--14.

\bibitem{connamacher2020uniformity}
{\sc H.~Connamacher and J.~Dobrosotskaya}, {\em On the uniformity of the approximation for $r$-associated {S}tirling numbers of the second kind}, Contributions to Discrete Mathematics, 15 (2020), pp.~25--42.

\bibitem{dobson1968note}
{\sc A.~J. Dobson}, {\em A note on {S}tirling numbers of the second kind}, Journal of Combinatorial Theory, 5 (1968), pp.~212--214.

\bibitem{knoblauch2008closed}
{\sc A.~Knoblauch}, {\em Closed-form expressions for the moments of the binomial probability distribution}, SIAM Journal on Applied Mathematics, 69 (2008), pp.~197--204.

\bibitem{knuth1997art}
{\sc D.~E. Knuth}, {\em The art of computer programming}, vol.~3, Pearson Education, 1997.

\bibitem{latala1997estimation}
{\sc R.~Lata{\l}a}, {\em Estimation of moments of sums of independent real random variables}, The Annals of Probability, 25 (1997), pp.~1502--1513.

\bibitem{sagan2020combinatorics}
{\sc B.~E. Sagan}, {\em Combinatorics: The art of counting}, vol.~210, American Mathematical Soc., 2020.

\bibitem{temme1993asymptotic}
{\sc N.~M. Temme}, {\em Asymptotic estimates of {S}tirling numbers}, Studies in Applied Mathematics, 89 (1993), pp.~233--243.

\bibitem{temme2014asymptotic}
{\sc N.~M. Temme}, {\em Asymptotic methods for integrals}, vol.~6, World Scientific, 2014.

\end{thebibliography}

\end{document}